\documentclass[amssymb, 11pt]{amsart}
\usepackage{latexsym}

\newlength{\standardunitlength}
\newtheorem{prop}{Proposition}[section]

\newtheorem{lemma}[prop]{Lemma}

\newtheorem{theorem}[prop]{Theorem}

\begin{document}

\title [Regular semisimple conjugacy classes] {The number of regular semisimple conjugacy classes in
the finite classical groups}

\author{Jason Fulman}
\address{Department of Mathematics\\
        University of Southern California\\
        Los Angeles, CA, 90089}
\email{fulman@usc.edu}

\author{Robert Guralnick}
\address{Department of Mathematics\\University of Southern California \\ Los Angeles, CA, 90089}
\email{guralnic@math.usc.edu}

\keywords{classical group, conjugacy class, regular semisimple, generating function}

\date{Version of April 26, 2012}

\begin{abstract} Using generating functions, we enumerate regular semisimple conjugacy classes in the finite classical groups.
For the general linear, unitary, and symplectic groups this gives a different approach to known results; for the special
orthogonal groups the results are new.
\end{abstract}

\maketitle

\section{Introduction}

An element of a finite classical group $G$ is said to be {\it semisimple} if its minimal polynomial has no repeated roots in the algebraic closure of $\mathbb{F}_q$. This is equivalent to being a $p'$ element, where $p$ is the characteristic of the field. An element of $G$ is said to be {\it regular} if its centralizer in the corresponding group over the algebraic closure of $\mathbb{F}_q$ has dimension equal to the Lie rank of the group. For the general linear, unitary, and symplectic groups, a matrix is regular if and only if its characteristic polynomial is equal to its minimal polynomial; this is not true for orthogonal groups.

The set of regular semisimple elements of a finite group of Lie type plays an important role in representation theory \cite{C}. There are many papers in the literature which have studied the enumeration of such elements. The papers \cite{FJ} and \cite{L1} give expressions enumerating regular semisimple elements; however the formulae are very complicated and it seems hard to use them to obtain simple upper and lower bounds. Guralnick and L\"{u}beck \cite{GL} show that for simple $G$, the proportion of regular semisimple elements is at least $1-3/(q-1)-2/(q-1)^2$. Neumann and Praeger \cite{NP} show that the proportion of regular semisimple matrices (not necessarily invertible) is at least $1-\frac{q^2}{(q^2-1)(q-1)}- \frac{1}{2}q^{-2} - \frac{2}{3}q^{-3}$ and is at most $1-q^{-1}+q^{-2}+q^{-3}$. Perhaps the most precise results are obtained using generating functions; the papers \cite{Fu} and \cite{W} use generating functions to show that for fixed $q$, the large $n$ proportion of regular semisimple elements of $GL(n,q)$ approaches $1-q^{-1}$ as $n \rightarrow \infty$. Wall \cite{W} gives error terms for the convergence rate to this limit. The memoir \cite{FNP} applies the generating function approach to obtain results for other finite classical groups.

Estimates on the proportion of regular semisimple elements played a crucial role in the solution of the Boston-Shalev conjecture stating that the proportion of fixed point free elements of a finite simple group in a transitive action on a finite set $X$ with $|X|>1$ is uniformly bounded away from zero \cite{FG}. We also note that the set of regular semisimple elements has been studied from the topological viewpoint by Lehrer \cite{L1} and Lehrer and Segal \cite{LS}.

In this paper we study the closely related problem of enumerating regular semisimple conjugacy classes in finite classical groups. Results for $GL(n,q)$ and $PGL(n,q)$ appear in \cite{L2}, and also in \cite{FJK}, which treats $SL(n,q), U(n,q), SU(n,q)$ as well. Fleischmann \cite{F} treats $PSL(n,q)$ and the groups $PSU(n,q)$, $Sp(2n,q)$, and $PSp(2n,q)$. Results for exceptional Lie groups are on the website \cite{Lu}. Lehrer \cite{L3} discusses the closely related problem of enumerating regular semisimple orbits on the Lie algebra. The problem of enumerating regular semisimple classes in orthogonal groups, to the best of our knowledge, is not treated in the literature, and the main point of the present paper is to fill this gap.
It turns out that in all cases, the answer has a simple form. For example in even characteristic, we show that for $n \geq 2$, the number of regular semisimple classes of $SO^{\pm}(2n,q)$ is \[ q^n-q^{n-1} \mp (-1)^n(q-1) .\] We argue using generating functions, and also show how this approach gives elementary derivations of the known results for $GL(n,q)$, $SL(n,q)$, $U(n,q)$, $SU(n,q)$, and $Sp(2n,q)$.

The organization of this paper is as follows. Section \ref{gl} uses generating functions to count the number of regular semisimple classes of $GL(n,q)$ and $SL(n,q)$. This is also carried out for $U(n,q)$ and $SU(n,q)$ in Section \ref{un} and for $Sp(2n,q)$ in Section \ref{sym}. Section \ref{orth} contains our main new results: simple formulae for the number of regular semisimple classes in the special orthogonal groups. We treat both odd and even characteristic.

\section{Linear groups} \label{gl}

In this section we derive a formula for the number of regular semisimple conjugacy classes of $GL(n,q)$ and $SL(n,q)$.
This result is already known from \cite{FJK}, but we argue using generating functions.

We elaborate a bit on the comments of the introduction.  If $G$ is equal to $GL(n,q)$, $U(n,q)$, $SL(n,q)$, $SU(n,q)$ or $Sp(2n,q)$,
an element is semisimple regular if and only if its minimal polynomial is equal to its characteristic polynomial which has
no multiple roots.  Moreover,  since these groups are simply connected (i.e. the derived subgroups of the corresponding
algebraic group are),  two semisimple elements of $G$ are conjugate if and only if they are conjugate in the algebraic group
if and only if they have the same characteristic polynomials.

Let $N(q;d)$ be the number of monic irreducible polynomials of degree $d$ over $\mathbb{F}_q$ with non-zero constant term. The following lemma is well known; see for instance part c of Lemma 1.3.10 of \cite{FNP}.

\begin{lemma} \label{lem1} \[ \prod_{d \geq 1} (1+u^d)^{-N(q;d)} = \frac{(1+u)(1-qu)}{(1-qu^2)}.\]
\end{lemma}

Theorem \ref{conjgl} counts regular semisimple conjugacy classes in $GL(n,q)$.

\begin{theorem} \label{conjgl} The number of regular semisimple conjugacy classes of the group $GL(n,q)$ is \[ \frac{q^{n+1}-q^n+(-1)^{n+1}(q-1)}{q+1} .\]
\end{theorem}

\begin{proof}   As we have seen
above,  the regular semisimple conjugacy classes of $GL(n,q)$ correspond to monic, degree $n$ squarefree polynomials, with non-0 constant term.
Hence the number of such classes is the coefficient of $u^n$ in \[ \prod_{d \geq 1} (1+u^d)^{N(q;d)}. \]
From Lemma \ref{lem1} one obtains that \[ \prod_{d \geq 1} (1+u^d)^{N(q;d)} = \frac{1-qu^2}{(1+u)(1-qu)} .\]
Taking coefficients of $u^n$ gives \[ \frac{q^{n+1}-q^n+(-1)^{n+1}(q-1)}{q+1} .\] Indeed,
\begin{eqnarray*}
& & 1+\sum_{n \geq 1} u^n \left[ \frac{q^{n+1}-q^n+(-1)^{n+1}(q-1)}{q+1} \right] \\
& = & 1+\frac{1}{q+1} \sum_{n \geq 1} [u^nq^{n+1} - u^nq^n - (-u)^n(q-1)] \\
& = & 1+\frac{1}{q+1} \left[ \frac{uq^2}{1-uq} - \frac{uq}{1-uq} + \frac{(q-1)u}{1+u} \right] \\
& = & \frac{1-qu^2}{(1+u)(1-qu)}.
\end{eqnarray*}
\end{proof}

{\it Remark}: The quantity in Theorem \ref{conjgl} is a polynomial in $q$, and can be written as
\[ q^n -2 (q^{n-1}- q^{n-2} + \cdots + q) + 1, \] for $n$ even, and as
\[ q^n -2 (q^{n-1} - q^{n-2} + \cdots - q) -1. \] for $n$ odd.

To treat the special linear groups, some further background is needed. We let $\Omega_{q-1}$ denote the $q-1$st roots of unity in $\mathbb{C}$, and let $\omega$ be an element of $\Omega_{q-1}$. Let $\zeta$ be a generator of the multiplicative group $\mathbb{F}_q^*$. For $\alpha \in \mathbb{F}_q^*$, define $r(\alpha)$ to be the element of $\mathbb{Z}_{q-1}$ such that $\zeta^{r(\alpha)} = \alpha$. For a polynomial $\phi$, define $r(\phi)=r((-1)^{deg(\phi)} \phi(0))$.

The following lemma of Britnell (Identity 3.5 of \cite{B}) will be helpful.

\begin{lemma} \label{brit} If $\omega \neq 1$, then
\[ \prod_{\phi} (1+\omega^{r(\phi)} u^{deg(\phi)}) = \left\{
\begin{array}{ll} 1 & \mbox{if $\omega \neq -1$} \\
\frac{1-qu^2}{1-u^2} & \mbox{if $\omega = -1$} \end{array} \right. \]
Here the product is over $\phi$ which are monic, irreducible polynomials over $\mathbb{F}_q$ satisfying $\phi(0) \neq 0$.
\end{lemma}

Next we enumerate the regular semisimple conjugacy classes of $SL(n,q)$.

\begin{theorem} \label{conjsl} The number of regular semisimple conjugacy classes of the group $SL(n,q)$ is
\[ \frac{q^{n+1}-q^n+(-1)^{n+1}(q-1)}{q^2-1} \] if $n$ is odd or $q$ is even, and is
\[ \frac{q^{n+1}-q^n - (q-1)}{q^2-1} - 1 \] if $n$ is even and $q$ is odd.
\end{theorem}

\begin{proof} The number of regular semisimple classes of $SL(n,q)$ is the number of monic, squarefree polynomials with constant
term $(-1)^n$. This is equal to the coefficient of $u^n$ in
\[ \frac{1}{q-1} \sum_{\omega \in \Omega_{q-1}} \prod_{\phi} (1+\omega^{r(\phi)} u^{deg(\phi)}) .\] Here the product is over $\phi$ which are monic, irreducible polynomials over $\mathbb{F}_q$ satisfying $\phi(0) \neq 0$.

Assume that $q$ is even. Then no $\omega \in \Omega_{q-1}$ is equal to $-1$, so Lemmas \ref{lem1} and \ref{brit} imply that the number of regular semisimple classes of $SL(n,q)$ is the coefficient of $u^n$ in \[ \frac{1}{q-1} \frac{1-qu^2}{(1+u)(1-qu)} + \frac{q-2}{q-1} .\] The first term corresponds to $\omega=1$ and the second term corresponds to $\omega \neq \pm 1$. By the proof of Theorem \ref{conjgl}, this is equal to $\frac{q^{n+1}-q^n+(-1)^{n+1}(q-1)}{q^2-1}$.

Next assume that $q$ is odd. Then $-1 \in \Omega_{q-1}$, so Lemmas \ref{lem1} and \ref{brit} imply that the number of regular semisimple classes of $SL(n,q)$ is the coefficient of $u^n$ in \[ \frac{1}{q-1} \frac{1-qu^2}{(1+u)(1-qu)} + \frac{1}{q-1} \frac{1-qu^2}{1-u^2} + \frac{q-3}{q-1}.\] The first term corresponds to $\omega=1$, the second term to $\omega =-1$, and the third term to $\omega \neq \pm 1$. If $n$ is odd, this is easily seen to be equal to $\frac{q^{n+1}-q^n+(-1)^{n+1}(q-1)}{q^2-1}$, and if $n$ is even, this is easily seen to be equal to $\frac{q^{n+1}-q^n - (q-1)}{q^2-1} - 1$.
\end{proof}

{\it Remark}: The quantity $\frac{q^{n+1}-q^n+(-1)^{n+1}(q-1)}{q^2-1}$ is equal to \[ q^{n-1}-q^{n-2}+q^{n-3} - \cdots + (-1)^n q - (-1)^n, \] and for $n$ even, the quantity $\frac{q^{n+1}-q^n - (q-1)}{q^2-1} - 1$ is equal to \[ q^{n-1}-q^{n-2}+q^{n-3} - \cdots + q - 2.\]

\section{Unitary groups} \label{un}

In this section we derive a formula for the number of regular semisimple classes of $U(n,q)$ and $SU(n,q)$. This was first done in \cite{FJK}, using different methods.

To begin we need some background on polynomials. The map $\sigma: x \mapsto x^q$ is an involutory automorphism of $\mathbb{F}_{q^2}$ and it induces an automorphism of the polynomial ring $\mathbb{F}_{q^2}[t]$ in an obvious way, namely $\sigma:\sum_{0 \leq i \leq n} a_i t^i \mapsto \sum_{0 \leq i \leq n} a_i^{\sigma} t^i$. An involutory map $\phi \mapsto \tilde{\phi}$ is defined on those monic polynomials $\phi \in \mathbb{F}_{q^2}[t]$ that have non-zero constant coefficient, by \[ \tilde{\phi}(t) = \phi(0)^{-\sigma} t^{deg(\phi)} \phi^{\sigma}(t^{-1}).\] Thus if
\[ \phi(t) = t^n +a_{n-1}t^{n-1} + \cdots + a_1t + a_0 \] with $a_0 \neq 0$, then its $\sim$-conjugate
is given by \[ \tilde{\phi}(t) = t^n + (a_1a_0^{-1})^{\sigma} t^{n-1} + \cdots + (a_{n-1}a_0^{-1})^{\sigma} t + (a_0^{-1})^{\sigma} .\]

We say that $\phi$ is self-conjugate (or $\sim$-self-conjugate) if $\phi(0) \neq 0$ and $\tilde{\phi}=\phi$. We let $\tilde{N}(q;d)$ denote the number of monic irreducible self-conjugate polynomials $\phi(t)$ of degree $d$ over $\mathbb{F}_{q^2}$, and let $\tilde{M}(q;d)$ denote the number of (unordered) conjugate pairs $\{ \phi,\tilde{\phi} \}$ of monic irreducible polynomials of degree $d$ over $\mathbb{F}_{q^2}$ that are not self conjugate. The following lemma, part b of Lemma 1.3.14 of \cite{FNP}, will be useful.

\begin{lemma} \label{lem2}
\[ \prod_{d \geq 1} (1+u^d)^{-\tilde{N}(q;d)} (1+u^{2d})^{-\tilde{M}(q;d)} = \frac{(1+u^2)(1-qu)}{(1+u)(1-qu^2)}.\]
\end{lemma}

Theorem \ref{countun} counts the regular semisimple conjugacy classes of $U(n,q)$.

\begin{theorem} \label{countun}
The number of regular semisimple classes of $U(n,q)$ is equal to
\[ \frac{(q+1) [q^{n+1}-q^n + (-1)^{n+1} (-1)^{\lfloor n/2 \rfloor} (q- (-1)^n)]}{q^2+1} .\]
\end{theorem}

\begin{proof} The semisimple conjugacy classes of $U(n,q)$ correspond to monic, degree n self-conjugate polynomials over $\mathbb{F}_{q^2}$ with non-zero constant term, and therefore the  regular semisimple conjugacy classes of $U(n,q)$ correspond to squarefree, monic, degree $n$ self-conjugate polynomials over $\mathbb{F}_{q^2}$ with non-zero constant term.

It follows that the number of regular semisimple conjugacy classes of $U(n,q)$ is the coefficient of $u^n$ in
\[ \prod_{d \geq 1} (1+u^d)^{\tilde{N}(q;d)} (1+u^{2d})^{\tilde{M}(q;d)} .\] By Lemma \ref{lem2}, this is the coefficient of $u^n$ in \[ \frac{(1+u)(1-qu^2)}{(1+u^2)(1-qu)}. \]

One computes that
\begin{eqnarray*}
& & 1 + \sum_{n \geq 1} u^n \left[ \frac{(q+1)(q^{n+1}-q^n + (-1)^{n+1} (-1)^{\lfloor n/2 \rfloor} (q- (-1)^n))}{q^2+1} \right] \\
& = & 1+\frac{q+1}{q^2+1} \sum_{n \geq 1} u^n(q^{n+1}-q^n) - \frac{q(q+1)}{q^2+1} \sum_{n \geq 1} u^n (-1)^n (-1)^{\lfloor n/2 \rfloor} \\
& & + \frac{q+1}{q^2+1} \sum_{n \geq 1} u^n (-1)^{\lfloor n/2 \rfloor}.
\end{eqnarray*} Since $\sum_{n \geq 1} u^n (-1)^n (-1)^{\lfloor n/2 \rfloor}= \frac{-u-u^2}{1+u^2}$ and $\sum_{n \geq 1} u^n (-1)^{\lfloor n/2 \rfloor} = \frac{u-u^2}{1+u^2}$ , this simplifies to

\begin{eqnarray*}
& & 1+ \frac{q+1}{q^2+1} \left[ \frac{uq^2}{1-uq} - \frac{uq}{1-uq} \right] - \frac{q(q+1)}{q^2+1} \left[ \frac{-u-u^2}{1+u^2} \right] + \frac{q+1}{q^2+1} \frac{u-u^2}{1+u^2} \\
& = & \frac{(1+u)(1-qu^2)}{(1+u^2)(1-qu)},
\end{eqnarray*} as desired.
\end{proof}

{\it Remark}: From Theorem \ref{countun}, one sees that the number of regular semisimple classes of $U(n,q)$ can be rewritten
\[ q^n -2 (q^{n-2} - q^{n-4} + \cdots + (-1)^{n/2} q^2) + (-1)^{n/2}\] if $n$ is even, and
\[ q^n - 2 (q^{n-2} - q^{n-4} + \cdots + (-1)^{(n-1)/2} q^3) + (-1)^{(n-1)/2} (2q+1) \] if $n \geq 3$ is odd.

To treat the case of $SU(n,q)$, some further definitions are needed. We let $\Omega_{q+1}$ denote the $q+1$st roots of unity in $\mathbb{C}$. We let $\zeta$ be a generator
of the cyclic subgroup of order $q+1$ in $\mathbb{F}_{q^2}^*$. For $\phi$ a monic irreducible polynomial over $\mathbb{F}_{q^2}$ with non-0 constant term, we define $s(\phi) \in \mathbb{Z}_{q+1}$ by
\[ \zeta^{s(\phi)} = \left\{
\begin{array}{ll} (-1)^{deg(\phi)} \phi(0) & \mbox{if $\phi= \tilde{\phi}$} \\
\phi(0) \tilde{\phi}(0) & \mbox{if $\phi \neq \tilde{\phi}$} \end{array} \right. \]

The following identity of Britnell (Identity 4.3 of \cite{B2}) will be helpful.

\begin{lemma} \label{brit2} For $\omega \neq 1 \in \Omega_{q+1}$, \[ \prod_{\phi=\tilde{\phi}} (1+\omega^{s(\phi)} u^{deg(\phi)}) \prod_{\{\phi,\tilde{\phi}\} \atop \phi \neq \tilde{\phi}}
(1+\omega^{s(\phi)} u^{2 deg(\phi)}) \] is equal to
\[ \left\{
\begin{array}{ll} 1 & \mbox{if $\omega \neq -1$} \\
\frac{1-qu^2}{1+u^2} & \mbox{if $\omega = -1$} \end{array} \right. \] Here $\phi$ ranges over monic irreducible polynomials over $\mathbb{F}_{q^2}$ with non-0 constant term.
\end{lemma}

Next we enumerate the regular semisimple conjugacy classes of $SU(n,q)$.

\begin{theorem} \label{conjsu} The number of regular semisimple conjugacy classes of the group $SU(n,q)$ is
\[ \frac{q^{n+1}-q^n+(-1)^{n+1} (-1)^{\lfloor n/2 \rfloor} (q-(-1)^n)}{q^2+1} \]
if $n$ is odd or $q$ is even, and is
\[ \frac{q^{n+1}-q^n-(-1)^{n/2}(q-1)}{q^2+1} + (-1)^{n/2} \]
if $n$ is even and $q$ is odd.
\end{theorem}

\begin{proof} The number of regular semisimple classes of $SU(n,q)$ is the number of squarefree characteristic polynomials of
elements of $SU(n,q)$.  This is the coefficient of $u^n$ in
\[ \frac{1}{q+1} \sum_{\omega \in \Omega_{q+1}} \prod_{\phi=\tilde{\phi}} (1+\omega^{s(\phi)} u^{deg(\phi)}) \prod_{\{\phi,\tilde{\phi}\} \atop \phi \neq \tilde{\phi}}
(1+\omega^{s(\phi)} u^{2 deg(\phi)}).\] Here $\phi$ ranges over monic irreducible polynomials over $\mathbb{F}_{q^2}$ with non-0 constant term.

Assume that $q$ is even. Then no $\omega \in \Omega_{q+1}$ is equal to $-1$, so Lemmas \ref{lem2} and \ref{brit2} imply that the number of regular semisimple
classes of $SU(n,q)$ is the coefficient of $u^n$ in
\[ \frac{1}{q+1} \frac{(1+u)(1-qu^2)}{(1+u^2)(1-qu)} + \frac{q}{q+1}.\] The first term corresponds to $\omega=1$ and the second term to $\omega \neq 1$. By the proof of Theorem \ref{countun}, this is $\frac{q^{n+1}-q^n+(-1)^{n+1} (-1)^{\lfloor n/2 \rfloor} (q-(-1)^n)}{q^2+1}$.

Next assume that $q$ is odd. Then $-1 \in \Omega_{q+1}$, so Lemmas \ref{lem2} and \ref{brit2} imply that the number of regular semisimple
classes of $SU(n,q)$ is the coefficient of $u^n$ in
\[ \frac{1}{q+1} \frac{(1+u)(1-qu^2)}{(1+u^2)(1-qu)} + \frac{1}{q+1} \frac{1-qu^2}{1+u^2} + \frac{q-1}{q+1} .\] The first term corresponds to $\omega=1$, the second term to $\omega= -1$, and the third term to $\omega \neq \pm 1$. This is
equal to $\frac{q^{n+1}-q^n+(-1)^{n+1} (-1)^{\lfloor n/2 \rfloor} (q-(-1)^n)}{q^2+1}$ for $n$ odd, and equal to
$\frac{q^{n+1}-q^n-(-1)^{n/2}(q-1)}{q^2+1} + (-1)^{n/2}$ if $n$ is even.
\end{proof}

\section{Symplectic groups} \label{sym}

This section uses generating functions to enumerate regular semisimple classes in $Sp(2n,q)$.

First we need some background on polynomials. For a monic polynomial $\phi(t) \in \mathbb{F}_q[t]$ of degree $n$ with non-zero constant term, we define the $*$-conjugate $\phi^*(t)$ by
\[ \phi^*(t)=\phi(0)^{-1} t^n \phi(t^{-1}).\] Thus if
\[ \phi(t) = t^n + a_{n-1}t^{n-1} + \cdots + a_1 t + a_0 \] then
\[ \phi^*(t) = t^n + a_1a_0^{-1} t^{n-1} + \cdots + a_{n-1}a_0^{-1} t + a_0^{-1}. \] We say that $\phi$ is self-conjugate (or $*$-self conjugate) if $\phi(0) \neq 0$ and $\phi^*=\phi$.

We let $N^*(q;d)$ denote the number of monic irreducible self-conjugate polynomials $\phi(t)$ of degree $d$ over $\mathbb{F}_{q}$, and let $M^*(q;d)$ denote the number of (unordered) conjugate pairs $\{ \phi,\phi^* \}$ of monic irreducible polynomials of degree $d$ over $\mathbb{F}_{q}$ that are not self conjugate. The following lemma, part b of Lemma 1.3.17 of \cite{FNP}, will be useful.

\begin{lemma} \label{lem3} Let $e=1$ if the characteristic is even and let $e=2$ if the characteristic is odd. Then
\[ \prod_{d \geq 1} (1+u^d)^{-N^*(q;2d)} (1+u^{d})^{-M^*(q;d)} = \frac{(1+u)^e (1-qu)}{1-qu^2}.\]
\end{lemma}

Theorem \ref{conjsym} is the main result of this section, and was proved in \cite{F} by different methods.

\begin{theorem} \label{conjsym} The number of regular semisimple conjugacy classes of the group $Sp(2n,q)$ is

\[  \begin{array}{ll}
\frac{q-1}{q+1} (q^n + (-1)^{n-1}) & \mbox{if q is even}\\
(-1)^n (n+1) + \sum_{i=0}^{n-1} (-1)^i (2i+1) q^{n-i} & \mbox{if q is odd}
                                                \end{array}
                                         \]
\end{theorem}

\begin{proof} The semisimple conjugacy classes of $Sp(2n,q)$ correspond to monic, degree 2n, self-conjugate polynomials over $\mathbb{F}_{q}$ with constant term 1, and the regular semisimple conjugacy classes of $Sp(2n,q)$ correspond to squarefree, monic, degree 2n self-conjugate polynomials over $\mathbb{F}_{q}$ with constant term 1. Note that by parity arguments, a regular semisimple conjugacy class of $Sp(2n,q)$ can not have a $z-1$ factor.

It follows that the number of regular semisimple conjugacy classes of $Sp(2n,q)$ is the coefficient of $u^n$ in
\[ \prod_{d \geq 1} (1+u^d)^{N^*(q;2d)} (1+u^{d})^{M^*(q;d)} .\] By Lemma \ref{lem3}, this is the coefficient of $u^n$ in \[ \frac{(1-qu^2)}{(1+u)^e(1-qu)}, \] where $e=1$ if $q$ is even, and $e=2$ if $q$ is odd.

Now if $q$ is even, comparing with the proof of Theorem \ref{conjgl} shows that this is the same as the generating function for regular semisimple conjugacy classes of $GL(n,q)$, so the result follows from Theorem \ref{conjgl}. If $q$ is odd, we need to compute the coefficient of $u^n$ in $\frac{1-qu^2}{(1+u)^2(1-qu)}$. Letting $[u^n]f(u)$ denote the coefficient of $u^n$ in a power series $f(u)$, we obtain that
\begin{eqnarray*}
[u^n] \frac{1}{(1+u)^2(1-qu)} & = & \sum_{j=0}^n [u^j] \frac{1}{(1+u)^2} [u^{n-j}] \frac{1}{1-qu} \\
& = & \sum_{j=0}^n (-1)^j (j+1) q^{n-j}.
\end{eqnarray*}
Thus
\begin{eqnarray*}
& & [u^n] \frac{1-qu^2}{(1+u)^2(1-qu)} \\
& = & [u^n] \frac{1}{(1+u)^2(1-qu)} - q [u^{n-2}] \frac{1}{(1+u)^2(1-qu)} \\
& = & \left( \sum_{j=0}^n (-1)^j (j+1) q^{n-j} \right) - q \left( \sum_{j=0}^{n-2} (-1)^j (j+1) q^{n-2-j} \right) \\
& = & (-1)^n (n+1) + \sum_{i=0}^{n-1} (-1)^i (2i+1) q^{n-i},
\end{eqnarray*} as needed.
\end{proof}

{\it Remark}: As noted in the proof of Theorem \ref{conjsym}, if $q$ is even then the number of regular semisimple conjugacy classes of $Sp(2n,q)$ is equal to the number of regular semisimple conjugacy classes of $GL(n,q)$. It would be nice to have a more direct proof of this.

\section{Orthogonal groups} \label{orth} This section contains our main new results. Subsection \ref{odd} enumerates regular semisimple classes in odd characteristic, and Subsection \ref{even} treats even characteristic.

\subsection{Odd characteristic} \label{odd}

In this subsection we suppose that $q$ is odd. Lemma \ref{oddchar} characterizes what it means for an element of $SO^{\pm}(n,q)$ to be regular semisimple.

\begin{lemma} \label{oddchar} An element of $SO^{\pm}(n,q)$ is regular semisimple if and only if:
\begin{itemize}
\item For all polynomials $\phi \neq z \pm 1$, the corresponding Jordan block partition has size at most 1
\item For $\phi = z + 1$, the Jordan block partition is either empty or consists of two blocks of size 1
\item For $\phi = z - 1$, the Jordan block partition is either empty, consists of one block of size 1, or consists of two blocks of size 1
\end{itemize}
\end{lemma}

\begin{proof}   We work in the corresponding algebraic group $X:=SO(n,\overline{\mathbb{F}_q})$ where $\overline{\mathbb{F}_q}$ is the algebraic closure of the field of order $q$.
Let $x \in X$ be a semisimple regular element.
If $\alpha \ne \pm 1$, then the $\alpha$ eigenspace of $x$ is totally singular and the sum of the $\alpha$ and
$\alpha^{-1}$ eigenspaces is nondegenerate.  If the multiplicity of the $\alpha$-eigenspace is $d$, then we see
that there is a subgroup $GL(d,\overline{\mathbb{F}_q})$ centralizing $x$, whence $d \le 1$.

Let $W$ denote the fixed space of $x$.  Then $W$ is nondegenerate and so $SO(W)$ is contained in the centralizer of $x$.
Since this must be a torus, it follows that $\dim W \le 2$.   The same argument applies to the $-1$ eigenspace.
Note that the $-1$ eigenspace must be even dimensional.

Conversely, if the conditions hold, we see that the connected part of the centralizer of $x$ is a torus.
\end{proof}

We also note that if the partition corresponding to $z+1$ or $z-1$ is nonempty, it can be of either positive or negative type.

Lemma \ref{dataforO} characterizes what it means for $x,y$ regular semisimple elements of $SO^{\pm}(n,q)$ to be conjugate in $O^{\pm}(n,q)$.

\begin{lemma}  \label{dataforO}  Let $x, y \in SO^{\pm}(n,q)$ be regular semisimple elements.   Then $x$ and $y$
are conjugate in $O^{\pm}(n,q)$ if and only if they have the same characteristic polynomial and the
$\pm 1$ eigenspaces have the same type.
\end{lemma}

\begin{proof}   Let $V$ denote the natural module for $O^{\pm}(n,q)$.  Let $f$ be the characteristic polynomial of $x$.
Write $f = (z-1)^a(z+1)^bg_3g_2 \ldots g_m$ where $g_i$ is either irreducible and self conjugate or $g_i$ is a product
of two irreducible polynomials which are conjugate to one another.   Let  $V_j$ denote the kernel of $g_j(x)$.  Then
each $V_j$ is nondegenerate.  If $g_j$ is irreducible,  $V_j$ is even dimensional of $-$ type and if $g_j$
is a product of two irreducible polynomials, then $V_j$ is even dimensional of $+$ type.   The result now follows
by Witt's theorem.
\end{proof}

Lemma \ref{oddsplit} describes which classes of $SO^{\pm}(n,q)$ are obtained by splitting a class of $O^{\pm}(n,q)$.

\begin{lemma} \label{oddsplit}
A conjugacy class of regular semisimple elements of $SO^{\pm}(n,q)$ is obtained by splitting a class of $O^{\pm}(n,q)$ if and only if its characteristic polynomial has no $z \pm 1$ factors.
\end{lemma}

\begin{proof}      Let $x \in C$ be an $O^{\pm}(n,q)$ class of semisimple regular elements contained in $SO^{\pm}(n,q)$.
Note that $C$ splits into two classes in  $SO^{\pm}(n,q)$ if and only if the centralizer of $x$ is contained in  $SO^{\pm}(n,q)$.
Suppose that $x$ commutes with an element $y \in O^{\pm}(n,q)$ with $\det(y)=-1$.  We may assume that $y$ is a $2$-element
and so in particular $y$ is semisimple.  It follows that the $-1$ eigenspace of $y$ is odd dimensional and nondegenerate.
Any odd dimensional nondegenerate $x$-subspace must have an eigenvalue $\pm 1$ (because the other eigenvalues
are paired on any nondegenerate invariant subspace).

Conversely, if $x$ has eigenvalue either $\pm 1$, it obviously has a nondegenerate $1$-dimensional space, whence it commutes
with a reflection.
\end{proof}

Let $N^*(q;d), M^*(q;d)$ be as in Section \ref{sym}. The following complement to Lemma \ref{lem3} will be useful. It is part c of Lemma 1.3.17 of \cite{FNP}.

\begin{lemma} \label{lem4} Suppose that the characteristic is odd. Then
\[ \prod_{d \geq 1} (1-u^d)^{-N^*(q;2d)} (1+u^{d})^{-M^*(q;d)} = \frac{(1-u)(1+u)^2}{1-qu^2}.\]
\end{lemma}

Let $rs_G$ denote the number of regular semisimple conjugacy classes of a group $G$. Define

\[ R_{SO^+}(u) = 1 + \sum_{n \geq 1} rs_{SO^+(2n,q)} u^n \]
\[ R_{SO^-}(u) = \sum_{n \geq 1} rs_{SO^-(2n,q)} u^n \]
\[ R_{SO}(u) = 1 + \sum_{n \geq 1} rs_{SO(2n+1,q)} u^n \]

The following lemma is crucial.

\begin{lemma} \label{ogenodd}  Suppose that the characteristic is odd.
\begin{enumerate}
\item \begin{eqnarray*}
& & R_{SO^+}(u^2) + R_{SO^-}(u^2) + 2u R_{SO}(u^2) \\
& = & \left[ (1+2u^2)(1+2u+2u^2) +1 \right] \frac{(1-qu^4)}{(1+u^2)^2(1-qu^2)} - 1.
\end{eqnarray*}

\item \[ R_{SO^+}(u^2) - R_{SO^-}(u^2) =  \frac{2(1-qu^4)}{(1+u^2)^2(1-u^2)} - 1 \]
\end{enumerate}
\end{lemma}

\begin{proof} For the first part, the term $(1+2u^2)$ comes from the polynomial $z+1$; the $1$ is for the empty partition, and the $2u^2$ is for the partition with two blocks of size 1, either of $+$ or $-$ type. The term $(1+2u+2u^2)$ comes from the polynomial $z-1$; the $1$ is for the empty partition, the $2u$ is for the partition of size 1 (either of $+$ or $-$ type), and the $2u^2$ is for the partition consisting of two blocks of size 1 (either of $+$ or $-$ type). Further note by Lemma \ref{oddsplit} that one must count twice those conjugacy classes with no $z \pm 1$ term. We conclude from Lemma \ref{dataforO} that
\begin{eqnarray*}
& & R_{SO^+}(u^2) + R_{SO^-}(u^2) + 2u R_{SO}(u^2) \\
& = & (1+2u^2)(1+2u+2u^2) \prod_{d \geq 1} (1+u^{2d})^{N^*(q;2d)} (1+u^{2d})^{M^*(q;d)} \\
& & + \prod_{d \geq 1} (1+u^{2d})^{N^*(q;2d)} (1+u^{2d})^{M^*(q;d)} - 1 \\
& = & \left[ (1+2u^2)(1+2u+2u^2) +1 \right] \frac{(1-qu^4)}{(1+u^2)^2(1-qu^2)} - 1.
\end{eqnarray*} the final equality being Lemma \ref{lem3}.

For part 2 of Lemma \ref{ogenodd}, note that due to cancelation of positive and negative types, the polynomials $z \pm 1$ both contribute a factor of 1 to $R_{SO^+}(u^2) - R_{SO^-}(u^2)$. Again one must count twice those conjugacy classes with no $z \pm 1$ term. Observe that the space corresponding to a self-conjugate polynomial $\phi \neq z \pm 1$ has negative type, whereas the space corresponding to a conjugate pair of non-self conjugate polynomials has positive type. Thus by Lemmas \ref{dataforO} and \ref{lem4},
\begin{eqnarray*}
R_{SO^+}(u^2) - R_{SO^-}(u^2) & = & 2 \prod_{d \geq 1} (1-u^{2d})^{N^*(q;2d)} (1+u^{2d})^{M^*(q;d)} - 1 \\
& = & \frac{2(1-qu^4)}{(1+u^2)^2(1-u^2)} - 1.
\end{eqnarray*}
\end{proof}

One can now solve for $R_{SO^+}(u), R_{SO^-}(u), R_{SO}(u)$.

Indeed, taking odd degree terms in part 1 of Lemma \ref{ogenodd} gives that
\[ 2u R_{SO}(u^2) = (1+2u^2) 2u \frac{(1-qu^4)}{(1+u^2)^2(1-qu^2)}, \]
which simplifies to
\[ R_{SO}(u^2) = \frac{(1+2u^2)(1-qu^4)}{(1+u^2)^2(1-qu^2)}. \]

Similarly, taking even degree terms in part 1 of Lemma \ref{ogenodd} gives that
\[ R_{SO^+}(u^2) + R_{SO^-}(u^2) =  \frac{[(1+2u^2)^2+1] (1-qu^4)}{(1+u^2)^2(1-qu^2)}  - 1 \]
Part 2 of Lemma \ref{ogenodd} gives that
\[ R_{SO^+}(u^2) - R_{SO^-}(u^2) = \frac{2(1-qu^4)}{(1+u^2)^2 (1-u^2)} - 1.\]

Thus \[ R_{SO^+}(u^2) = \frac{(1+2u^2+2u^4)(1-qu^4)}{(1+u^2)^2(1-qu^2)} + \frac{(1-qu^4)}{(1+u^2)^2(1-u^2)} - 1 \] and
\[ R_{SO^-}(u^2) = \frac{(1+2u^2+2u^4)(1-qu^4)}{(1+u^2)^2(1-qu^2)} - \frac{(1-qu^4)}{(1+u^2)^2(1-u^2)} \]

We now obtain the main result of this section.

\begin{theorem} Assume that the characteristic is odd.
\begin{enumerate}
\item The number of regular semisimple classes of $SO(3,q)$ is $q$, and the number of regular semisimple classes of $SO(5,q)$ is $q^2-q-1$. For $n \geq 3$, the number of regular semisimple classes of $SO(2n+1,q)$ is
    \[ q^n-q^{n-1}-q^{n-2}+3q^{n-3}-5q^{n-4}+7q^{n-5} + \cdots + (-1)^n (2n-5) q - (-1)^n (n-1).\]

\item The number of regular semisimple classes of $SO^+(4,q)$ is $q^2-2q+3$. For $n \geq 4$ even, the number of regular semisimple conjugacy classes of $SO^+(2n,q)$ is
    \[ q^n-q^{n-1}+q^{n-2}-3q^{n-3}+5q^{n-4} - \cdots + (2n-7)q^2 - \left( \frac{5n-10}{2} \right) q + \frac{3n}{2}.\]

\item The number of regular semisimple classes of $SO^+(2,q)$ is $q-1$ and the number of regular semisimple classes of $SO^+(6,q)$ is $q^3-q^2+2q-4$. For $n \geq 5$ odd, the number of regular semisimple conjugacy classes of $SO^+(2n,q)$ is \[ q^n-q^{n-1}+q^{n-2}-3q^{n-3}+5q^{n-4} - \cdots - (2n-7)q^2 + \left( \frac{5n-11}{2} \right) q - \frac{3n-1}{2}.\]

\item The number of regular semisimple classes of $SO^-(4,q)$ is $q^2-1$. For $n \geq 4$ even, the number of regular semisimple conjugacy classes of $SO^-(2n,q)$ is
    \[ q^n-q^{n-1}+q^{n-2}-3q^{n-3}+5q^{n-4} - \cdots + (2n-7)q^2 - \left( \frac{3n-10}{2} \right) q + \frac{n-4}{2}.\]

\item The number of regular semisimple classes of $SO^-(2,q)$ is $q+1$ and the number of regular semisimple classes of $SO^-(6,q)$ is $q^3-q^2$. For $n \geq 5$ odd, the number of regular semisimple classes of $SO^-(2n,q)$ is
   \[ q^n-q^{n-1}+q^{n-2}-3q^{n-3}+5q^{n-4} - \cdots - (2n-7)q^2 + \left( \frac{3n-9}{2} \right) q - \frac{n-3}{2}.\]
\end{enumerate}
\end{theorem}

\begin{proof} Let $[u^n] f(u)$ denote the coefficient of $u^n$ in a power series $f(u)$. For the first part of the theorem, the generating function for $R_{SO}(u^2)$ gives that number of regular semisimple classes of $SO(2n+1,q)$ is equal to
\[ [u^n] \frac{(1+2u)(1-qu^2)}{(1+u)^2(1-qu)} = [u^n] \frac{(1-qu^2)}{(1+u)^2(1-qu)} + 2 [u^{n-1}] \frac{(1-qu^2)}{(1+u)^2(1-qu)} .\]
By the proof of the symplectic group case (Theorem \ref{conjsym}),
\begin{eqnarray*}
& & [u^n] \frac{(1-qu^2)}{(1+u)^2(1-qu)} \\
& = & q^n-3q^{n-1}+5q^{n-2}-7q^{n-3} + \cdots + (-1)^{n-1}(2n-1)q + (-1)^n(n+1)
\end{eqnarray*} and
\begin{eqnarray*}
& & 2 [u^{n-1}] \frac{(1-qu^2)}{(1+u)^2(1-qu)} \\
& = & 2q^{n-1}-6q^{n-2}+10q^{n-3} + \cdots + 2(-1)^{n-2}(2n-3)q + (-1)^{n-1}2n.
\end{eqnarray*} Addition now proves the first part of the theorem.

For the second part of the theorem, we assume $n$ is even, and need to evaluate
\[ [u^n] \left( \frac{(1+2u+2u^2)(1-qu^2)}{(1+u)^2(1-qu)} + \frac{(1-qu^2)}{(1+u)^2(1-u)} \right).\]
We write this as a sum of three terms:
\[ [u^n] \frac{(1+2u)(1-qu^2)}{(1+u)^2(1-qu)} + 2 [u^{n-2}] \frac{(1-qu^2)}{(1+u)^2(1-qu)} + [u^n] \frac{(1-qu^2)}{(1+u)^2(1-u)}.\]
By part one of the theorem, the first term is equal to
\[ q^n-q^{n-1}-q^{n-2}+3q^{n-3}-5q^{n-4}+7q^{n-5} + \cdots + (2n-5) q - (n-1).\]
By the odd characteristic case of Theorem \ref{conjsym}, the second term is equal to
\[ 2 \left( q^{n-2}-3q^{n-3}+5q^{n-4} - 7q^{n-5} \cdots -(2n-5)q + (n-1) \right).\]
The third term is easily seen to equal $-\frac{n}{2} q + \left( \frac{n}{2}+1 \right)$. Adding these three terms
proves part 2 of the theorem.

For the third part of the theorem, we assume $n$ is odd, and need to evaluate
\[ [u^n] \left( \frac{(1+2u+2u^2)(1-qu^2)}{(1+u)^2(1-qu)} + \frac{(1-qu^2)}{(1+u)^2(1-u)} \right).\] Again we write this as a sum of
three terms:
\[ [u^n] \frac{(1+2u)(1-qu^2)}{(1+u)^2(1-qu)} + 2 [u^{n-2}] \frac{(1-qu^2)}{(1+u)^2(1-qu)} + [u^n] \frac{(1-qu^2)}{(1+u)^2(1-u)}.\]
By part one of the theorem, the first term is equal to
\[ q^n-q^{n-1}-q^{n-2}+3q^{n-3}-5q^{n-4}+7q^{n-5} + \cdots - (2n-5) q + (n-1).\]
By the odd characteristic case of Theorem \ref{conjsym}, the second term is equal to
\[ 2 \left( q^{n-2}-3q^{n-3}+5q^{n-4} - 7q^{n-5} \cdots + (2n-5)q - (n-1) \right).\] The third term is
easily seen to be $\frac{n-1}{2} q - \left( \frac{n+1}{2} \right)$. Adding the three terms proves part three of
the theorem.

To prove part 4 of the theorem, by part 2 of Lemma \ref{ogenodd}, it is enough to verify that for even $n$, the number of regular semisimple classes in $SO^+(2n,q)$ minus the number of regular semisimple classes in $SO^-(2n,q)$ is the coefficient of $u^n$ in \[ \frac{2(1-qu^2)}{(1+u)^2(1-u)}.\] For even $n$, this coefficient is easily proved to equal $-nq + (n+2)$, which indeed is the difference between the expressions in parts 2 and 4 of the theorem.

To prove part 5 of the theorem, by part 2 of Lemma \ref{ogenodd}, it is enough to verify that for odd $n$, the number of regular semisimple classes in $SO^+(2n,q)$ minus the number of regular semisimple classes in $SO^-(2n,q)$ is the coefficient of $u^n$ in \[ \frac{2(1-qu^2)}{(1+u)^2(1-u)}.\] For odd $n$, this coefficient is easily seen to equal $(n-1)q-(n+1)$, which is the difference between the expressions in parts 3 and 5 of the theorem.
\end{proof}

\subsection{Even characteristic} \label{even}

Next suppose that $q$ is even. We need only consider $SO^{\pm}(2n,q)$, as $SO(2n+1,q)$ is isomorphic to a $Sp(2n,q)$ and $g \in SO(2n+1)$ is regular semisimple if and only if the corresponding element of $Sp(2n,q)$ is regular semisimple.

Lemma \ref{evenchar} characterizes the regular semisimple elements of $SO^{\pm}(2n,q)$.
The proof is identical to the case of odd characteristic.

\begin{lemma} \label{evenchar} An element of $SO^{\pm}(2n,q)$ is regular semisimple if and only if:
\begin{itemize}
\item For all polynomials $\phi \neq z - 1$, the corresponding Jordan block partition has size at most 1
\item For $\phi = z - 1$, the Jordan block partition is either empty or consists of two blocks of size 1
\end{itemize}
\end{lemma}

Lemma \ref{dataforO2} characterizes what it means for $x,y$ regular semisimple elements of $SO^{\pm}(n,q)$ to be conjugate in $O^{\pm}(n,q)$.

\begin{lemma}  \label{dataforO2}  Let $x, y \in SO^{\pm}(n,q)$ be regular semisimple elements.   Then $x$ and $y$
are conjugate in $O^{\pm}(n,q)$ if and only if they have the same characteristic polynomial.
\end{lemma}

\begin{proof}   Let $\tau$ be an element in $O^{\pm}(n,q)$ outside $SO^{\pm}(n,q)$.
Let $\overline{\mathbb{F}_q}$ be the algebraic closure of the field of $q$ elements.  Since $SO(n,q)$ is simply connected,
the centralizers of semisimple elements are connected and so $x$ and $y$ are conjugate in
$SO^{\pm}(n,q)$ if and only if they are conjugate in $SO(n,\overline{\mathbb{F}_q})$ (and similarly for $x$ and $y^{\tau}$).

So suppose that $x$ and $y$ have the same characteristic polynomial.  Let $V$ be the natural module
for $SO(n,\overline{\mathbb{F}_q})$.  So $V = V_1 \oplus V_2$ where each $V_i$ is isotropic
and the characteristic polynomial of $x$ on $V_i$ is $f_i$ (and the $f_i$ are conjugate to one another).
We can similarly decompose $V=W_1 \oplus W_2$ for $y$ (with the same characteristic polynomials).
Since pairs of complementary totally isotropic subspaces are in a single $O(n,\overline{\mathbb{F}_q})$ orbit (there are two orbits
for $SO(n,\overline{\mathbb{F}_q})$), we may conjugate $y$ by an element of $O(n,\overline{\mathbb{F}_q})$ and
assume that  $W_i=V_i$.   The stabilizer of this pair of subspaces is
a subgroup $H \cong GL(V_1)$.  Then conjugating by an element of $H$ we may assume that $x$ and $y$
agree on $V_1$ and thus also on $V_2$.
\end{proof}

Lemma \ref{evensplit} describes which classes of $SO^{\pm}(n,q)$ are obtained by splitting a class of $O^{\pm}(n,q)$.

\begin{lemma} \label{evensplit}
A conjugacy class of regular semisimple elements of $SO^{\pm}(n,q)$ is obtained by splitting a class of $O^{\pm}(n,q)$ if and only
if its characteristic polynomial has no $z-1$ factor.
\end{lemma}

\begin{proof}   Let $x \in C \subset SO^{\pm}(n,q)$ with $C$ a semisimple conjugacy class of $O^{\pm}(n,q)$.
If the $1$-eigenspace is nontrivial, then it is even dimensional and nondegenerate.  Thus,
$x$ commutes with a transvection $y$.  Since $SO^{\pm}(n,q)$ contains no transvections, the
centralizer of $x$ is not contained in $SO^{\pm}(n,q)$, whence $C$ is a single conjugacy class
of $SO^{\pm}(n,q)$.

Conversely,  if the $1$-eigenspace of $x$ is trivial, then we see that the centralizer of $x$ in $GL(n,q)$
is a torus and in particular has odd order (and so is contained in $SO^{\pm}(n,q)$).
\end{proof}

The following lemma, part c of Lemma 1.3.17 of \cite{FNP}, will be useful.

\begin{lemma} \label{lem5} Suppose that the characteristic is even. Then
\[ \prod_{d \geq 1} (1-u^d)^{-N^*(q;2d)} (1+u^{d})^{-M^*(q;d)} = \frac{1+u}{1-qu^2}.\]
\end{lemma}

As in the odd characteristic section, we define generating functions:

\[ R_{SO^+}(u) = 1 + \sum_{n \geq 1} rs_{SO^+(2n,q)} u^n \]
\[ R_{SO^-}(u) = \sum_{n \geq 1} rs_{SO^-(2n,q)} u^n \]

Lemma \ref{ogeneven} solves for the generating functions $R_{SO^+}(u), R_{SO^-}(u)$.

\begin{lemma} \label{ogeneven} Suppose that the characteristic is even.
\begin{enumerate}
\item \[ R_{SO^+}(u^2) = \frac{1-qu^4}{1-qu^2} + \frac{1-qu^4}{1+u^2} - 1.\]

\item \[ R_{SO^-}(u^2) = \frac{1-qu^4}{1-qu^2} - \frac{1-qu^4}{1+u^2} . \]
\end{enumerate}
\end{lemma}

\begin{proof} It suffices to prove the two equalities:
\begin{itemize}
\item \[ R_{SO^+}(u^2) + R_{SO^-}(u^2) = \frac{2(1-qu^4)}{(1-qu^2)} - 1, \]
\item \[ R_{SO^+}(u^2) - R_{SO^-}(u^2) = \frac{2(1-qu^4)}{(1+u^2)} - 1.\]
\end{itemize}
To calculate $R_{SO^+}(u^2) + R_{SO^-}(u^2)$, there is a term $1+2u^2$ coming from the polynomial $z-1$, the $2u^2$ corresponding to the fact that the partition consisting of two blocks of size 1 can have either positive or negative type. Further note by Lemma \ref{evensplit} that we need to count twice those conjugacy classes with no $z - 1$ term. We conclude from Lemma \ref{dataforO2} that
\begin{eqnarray*}
R_{SO^+}(u^2) + R_{SO^-}(u^2) & = & (1+2u^2) \prod_{d \geq 1} (1+u^{2d})^{N^*(q;2d)} (1+u^{2d})^{M^*(q;d)} \\
& & + \prod_{d \geq 1} (1+u^{2d})^{N^*(q;2d)} (1+u^{2d})^{M^*(q;d)} - 1.
\end{eqnarray*} By Lemma \ref{lem3}, this is equal to
\begin{eqnarray*}
&  & \frac{(1+2u^2)(1-qu^4)}{(1+u^2)(1-qu^2)} + \frac{(1-qu^4)}{(1+u^2)(1-qu^2)} - 1\\
& = & \frac{2(1-qu^4)}{(1-qu^2)} - 1,
\end{eqnarray*} as claimed.

To calculate $R_{SO^+}(u^2) - R_{SO^-}(u^2)$, the polynomial $z-1$ contributes $1+u^2-u^2=1$. Further note by Lemma \ref{evensplit} that we need to count twice those conjugacy classes with no $z - 1$ term. Observe that the space corresponding to a self-conjugate polynomial $\phi \neq z-1$ has negative type, whereas the space corresponding to a conjugate pair of non-self conjugate polynomials has positive type. Hence by Lemma \ref{dataforO2},
\[ R_{SO^+}(u^2) - R_{SO^-}(u^2) = 2 \prod_{d \geq 1} (1-u^{2d})^{N^*(q;2d)} (1+u^{2d})^{M^*(q;d)} -1. \] By Lemma \ref{lem5}, this is equal to \[ \frac{2(1-qu^4)}{(1+u^2)} - 1, \] as claimed.
\end{proof}

Applying Lemma \ref{ogeneven} gives the main result of this section.

\begin{theorem} Assume that the characteristic is even. The number of regular semisimple classes of $SO^{\pm}(2,q)$ is $q \mp 1$, and for $n \geq 2$, the number of regular semisimple classes of $SO^{\pm}(2n,q)$ is \[ q^n-q^{n-1} \mp (-1)^n(q-1) .\]
\end{theorem}

\begin{proof} Consider the case of $SO^+(2n,q)$. Letting $[u^n] f(u)$ denote the coefficient of $u^n$ in a power series $f(u)$, Lemma \ref{ogeneven} gives that the number of regular semisimple classes of $SO^+(2n,q)$ is
\begin{eqnarray*}
& & [u^n] \left( \frac{1-qu^2}{1-qu} + \frac{1-qu^2}{1+u} \right) \\
& = & [u^n] \frac{1}{1-qu} + [u^{n-2}] \frac{-q}{1-qu} + [u^n] \frac{1}{1+u} + [u^{n-2}] \frac{-q}{1+u}\\
& = & q^n-q^{n-1} - (-1)^n(q-1).
\end{eqnarray*} The argument for $SO^-(2n,q)$ is nearly identical.
\end{proof}

\section{Acknowledgements} Fulman was supported by NSF grant DMS 0802082 and Guralnick was supported by NSF grant DMS 1001962.
Fulman and Guralnick also thank the Simons Foundation for its support.  We thank Frank L\"{u}beck for helpful correspondence.

\end{document}